\documentclass[10pt]{amsart} 
\usepackage{amsfonts,amssymb,amscd,amsmath,enumerate,verbatim,calc} 
 
\setbox0=\hbox{$+$}
\newdimen\plusheight
\plusheight=\ht0
\def\+{\;\lower\plusheight\hbox{$+$}\;}

\setbox0=\hbox{$-$}
\newdimen\minusheight
\minusheight=\ht0
\def\-{\;\lower\minusheight\hbox{$-$}\;}

\setbox0=\hbox{$\cdots$}
\newdimen\cdotsheight
\cdotsheight=\plusheight
\def\cds{\lower\cdotsheight\hbox{$\cdots$}}

\setbox0=\hbox{$+$}

\newtheorem{thm}{Theorem}[section]
\newtheorem{lem}[thm]{Lemma}
\newtheorem{cor}[thm]{Corollary}
\newtheorem{rem}[thm]{Remark}

\begin{document}
\title[Binomial Symbols and Prime Moduli] {Binomial Symbols and Prime Moduli}
\author{Manjil P.~Saikia}
\address{Department of Mathematical Sciences, Tezpur University, Napaam, Sonitpur, Pin-784028, India}
\email{manjil.saikia@gmail.com, manjil\_msi09@agnee.tezu.ernet.in}
\author{Jure Vogrinc}
\address{Faculty of Mathematics and Physics, University of Ljubljana, Jadranska ul.~19, 1000, Ljubljana, Slovenia}
\email{jure.vogrinc@gmail.com}


\vspace*{0.5in}
\begin{center}
{\bf Binomial Symbols and Prime Moduli}\\[5mm]
{\footnotesize  MANJIL P.~SAIKIA\footnote{Corresponding Author: manjil.saikia@gmail.com} and JURE VOGRINC}\\[3mm]
\end{center}

\vskip 5mm \noindent{\footnotesize{\bf Abstract.} We try to improve a problem asked in an Indian Math Olympiad. We give a brief overview of the work done in \cite{mps} and \cite{mps2} where the authors have found a periodic sequence and the length of its period, all inspired from an Olympiad problem. The main goal of the paper is to improve the main result in \cite{mps}.}

\vskip 3mm

\noindent{\footnotesize Key Words: Prime Moduli, Binomial Co-efficients, Periodic sequences, Primality Testing.}

\vskip 3mm

\noindent{\footnotesize 2000 Mathematical Reviews Classification
Numbers: 11A07, 11A41, 11A51, 11B50.}

\section{{Motivation}}

The motivation behind this work was an Indian Olympiad problem mentioned in \cite{mps}. Saikia and Vogrinc in \cite{mps} have proved the following result.

\begin{thm}
A natural number $p>1$ is a prime if and only if $\binom{n}{p}-\lfloor\frac{n}{p}\rfloor$ is divisible by $p$ for every non-negative $n$, where $\binom{n}{p}$ is the number of different ways in which we can choose $p$ out of $n$ elements and $\lfloor x \rfloor$ is the greatest integer not exceeding the real number $x$.
\end{thm}

We give three different proofs of the above result in \cite{mps}. For the sake of completeness we give below a proof.

\begin{proof}
First assume that $p$ is prime. Now we consider $n$ as $n=ap+b$ where $a$ is a non-negative integer and $b$ an integer $0\leq b<p$. Obviously,

$$\lfloor \frac{n}{p}\rfloor=\lfloor \frac{ap+b}{p}\rfloor\equiv a~(mod~p).$$

Now let us calculate $\binom{n}{p}~(mod~p)$.

$$\binom{n}{p}=\binom{ap+b}{p}$$
$$=\frac{(ap+b)\cdot(ap+b-1)\cdots(ap+1)\cdot ap\cdot(ap-1)\cdots(ap+b-p+1)}{p\cdot(p-1)\cdots 2\cdot1}$$
$$=\frac{a\cdot(ap+b)\cdot(ap+b-1)\cdots(ap+1)\cdot(ap-1)\cdots(ap+b-p+1)}{(p-1)\cdot(p-2)\cdots 2\cdot 1}$$

We denote this number by $X$.

We have $X\equiv c~(mod~p)$ for some $0\leq c<p$. Consequently taking modulo $p$, we have
 $$c(p-1)!=X(p-1)!=a(ap+b)\cdots(ap+1)(ap-1)\cdots(ap+b-p+1)$$

All the numbers $ap+b,\dots ,ap+b+1-p$ (other than $ap$) are relatively prime to $p$ and obviously none differ more than $p$ so they make a reduced residue system modulo $p$, meaning we have mod $p$,
$$(p-1)!=(ap+b)\cdot(ap+b-1)\dots(ap+1)\cdot(ap-1)\cdot(ap+b-p+1)$$
both sides of the equation being relatively prime to $p$ so we can deduce $X\equiv c \equiv a~(mod~p)$. And finally  $\binom{n}{p}\equiv X\equiv a\equiv \lfloor \frac{n}{p}\rfloor~(mod~p)$.

To complete the other part of the theorem we must construct a counterexample for every composite number $p$. If $p$ is composite we can consider it as $q^x\cdot k$ where $q$ is some prime factor of $p$, $x$ its exponent and $k$ the part of $p$ that is relatively prime to $q$ ($x$ and $k$ cannot be simultaniously $1$ or p is prime). We can obtain a counterexample by taking $n=p+q=q^xk+q$ will make a counter example. We have: $$\binom{p+q}{p}=\binom{p+q}{q}=\frac{(q^xk+q)(q^xk+q-1)\dots (q^xk+1)}{q!}$$
Which after simplifying the fraction equals: $(q^{x-1}k+1)\frac{(q^xk+q-1)\dots (q^xk+1)}{(q-1)!}$. Similary as above we have $(q^xk+q-1)\dots (q^xk+1)=(q-1)!\neq 0$ modulo $q^x$ therefore,
$$\frac{(q^xk+q-1)\dots (q^xk+1)}{(q-1)!}\equiv 1~(mod~q^x)$$ and
 $$\binom{p+q}{p}\equiv q^{x-1}k+1~(mod~q^x).$$

On the other hand obviously,
$$\lfloor\frac{q^xk+q}{q^xk}\rfloor\equiv 0~(mod~q^x).$$

And since $q^{x-1}k+1$ can never be equal to $0$ modulo $q^x$ we see that $$\binom{p+q}{p}\neq \lfloor\frac{p+q}{p}\rfloor~(mod~q^x)$$
consequently also incongruent modulo $p=q^xk$.

\end{proof}

\begin{rem}
Here we would like to comment that by taking $q$ as the minimal prime factor of $p$ and using the same method as above we can simplify the proof even more. We can than compare $\lfloor\frac{p+q}{p}\rfloor$ and $\binom{p+q}{p}$ directly modulo $p=q^xk$ and not $q^x$.
\end{rem}

\begin{rem} Instead of looking modulo $p$, we can look at higher powers of $p$, or we can look at the $n$-th Fibbonacci prime and so on. However, initial investigations by the authors suggest that finding a congruence relation in those cases becomes more difficult.
\end{rem}

This theorem is the motivation behind the following two theorems.

\begin{thm}
The sequence $a_n=\binom{m}{x}~(mod~m)$ is periodic, where $x,m \in \mathbb{N}$.
\end{thm}

The proof based on mathematical induction can be found in \cite{nehu} and \cite{mps2}. We present below a slightly modified account.

\begin{proof}
If $x=1$ the sequence is obviously periodic for any modulo $m$.

Now we assume that the sequence is periodic for a fixed $x$ and arbitrary $m$. We note that

$$\binom{n}{x+1}=\sum^{n-1}_{i=1}\binom{i}{x}.$$

Let $k$ be the length of a period of sequence $a_n=\binom{n}{x}~(mod~m)$, meaning $\binom{n+k}{x}\equiv \binom{n}{x}~(mod~m)$.

Therefore $\sum^k_{i=1} \binom{i}{x}\equiv c~(mod~m)$ for some $c$ and consequently $\sum^{n+mk}_{i=n+1} \binom{i}{x}=mc=0~(mod~m)$ for every integer $n$. All that is now required is another calclulation (the second equality from the right is modulo $m$):
$$\binom{n+mk}{x+1}=\sum^{n+mk-1}_{i=1}\binom{i}{x}=\sum^{n-1}_{i=1}\binom{i}{x}+\sum^{n+mk-1}_{i=n}\binom{i}{x}=\sum^{n-1}_{i=1}\binom{i}{x}=\binom{n}{x+1}$$
This now shows that sequence $b_n=\binom{n}{x+1}~(mod~m)$ is also periodic for every modulo $m$ which completes the induction and yields the desired result.
\end{proof}

\begin{rem}
Because of the upper result we know that a sequence $a_n=\binom{n}{m}~(mod~m)$ is also periodical. And a sequence $c_n=\lfloor\frac{n}{m}\rfloor~(mod~m)$ is obviously periodical. This combined with the above yields that for a composite modulo $m$ there exist infinitely many natural numbers $n$ such that $a_n\neq c_n~(mod~m)$.
\end{rem}

The above theorem states that for every $m$ the sequence $a_n=\binom{n}{m}~(mod~m)$ is periodic. The next most natural question to ask is, what is the minimal length of the period? Which gives us,
\begin{thm}
For a natural number $m=\prod^k_{i=1}p_i^{b_i}$, the sequence $a_n=\binom{n}{m}~(mod~m)$ has a period of minimal length,

$$l(m)=\prod^k_{i=1}p_i^{\lfloor \log_{p_i}m \rfloor+b_i}$$
\end{thm}

The proof given below is the one given by the authors in \cite{mps2}.

\begin{proof}

A sequence $a_n=\binom{n}{m}~(mod~m)$ where $m=\prod^k_{i=1}p_i^{b_i}$ starts with $m$ zeroes (we start with $a_0$). Now let us see when is the next time we have $m$ consecutive zeroes in the sequence $a_n$. We assume this happenes at some natural number $n$, that is
$$\binom{n}{m}\equiv \binom{n+1}{m}\equiv \dots\equiv \binom{n+m-1}{m}\equiv 0~(mod~m).$$

Let $p$ be a prime dividing $m$ and $b$ be it's exponent in the prime factorisation of $m$. We have $\binom{n+i}{m}\equiv 0~(mod~p^b)$ for $0\leq i<m$.

Obviously the exponent of $p$ in prime factorisation of $m!$ is $$\vartheta_p(m)=\sum^{\infty}_{i=1}\lfloor\frac{m}{p^i}\rfloor=\sum^{k}_{i=1}\lfloor\frac{m}{p^i}\rfloor,$$ where $k$ is the last summand different to zero and $k=\lfloor\log_p m\rfloor$.

Among numbers $n+1,n+2\dots,n+m$ there exist one that is divisible by $p^k$ (there are $m$ consecutive numbers and $m\geq p^k$). We denote this number by $x$. We have,
$$\binom{x-1}{m}=\frac{(x-1)(x-2)\dots(x-m)}{m!}.$$

Since we have $-(x-i)\equiv i~(mod~p^j)$ for all $1\leq i<m$ and $1\leq j\leq k$, so there are same number of numbers divisible by $p^j$ in $(x-1)(x-2)\dots(x-m)$ as in $m!$ for $1\leq j\leq k$.

On the other hand we have $\binom{x-1}{m}\equiv 0~(mod~p^b)$ (since $x-1$ is one of the numbers $n,n+1\dots n+m-1$). Of $m$ consecutive integers obviously only one can be divisible by $p^j$ if $j<k$. Therefore if we want the numerator of $\binom{x-1}{m}$ to have exponent of $p$ for $b$ larger than the denominator (that is in order to have $\binom{x-1}{m}\equiv 0~(mod~p^b)$) we need one of the numbers of the nominator to be divisible by $p^{\lfloor\log_p m\rfloor+a}$. Denote this number by $y$.

We assume $y\neq n$. Than either $y+m$ (if $y<n$) or $y-1$ (if $y>n$) are in the set $n,n+1\dots n+m-1$. This means that $$\binom{y+m}{m}\equiv 0~(mod~p^b).$$
(The other case is very similar and uses the same argument.)

However that is imposible since $y\equiv 0~(mod~p^{k+1})$ meaning the exponent of $p$ in prime factorisation of $(y+m)(y+m-1)\dots(y+1)$ is the same as in prime factorisation of $m!$ or in other words that $\binom{y+m}{m}$ is relatively prime to $p$. We reached a contradiction which means $y=n$.

The same argument will work for any arbitrary prime number dividing $m$. That means for every prime number $p$ dividing $m$ (infact $p^b|m$) we need $n$ to be divisible by $p^{\lfloor\log_p m\rfloor+b}$, therefore the length of the period of the sequence, $a_n$ must be a multiple of the number $\prod^k_{i=1}p_i^{\lfloor \log_{p_i}m \rfloor +b_i}$.

All that remains is to show that this infact is the lenght of the period. We need to prove that for every natural number $n$ we have $$\binom{n}{m}\equiv\binom{n+l(m)}{m}~(mod~m),$$ where $$l(m)=\prod^k_{i=1}p_i^{\lfloor \log_{p_i}m \rfloor+b_i}.$$

Because of some basic properties of congruences ($a\equiv b~(mod~m)$ equivalent to $ax\equiv bx~(mod~m)$ if $\gcd(m,x)=1$), it is enough to show that, $$\frac{\prod^{m-1}_{i=0}(n-i)}{\prod^k_{i=1}p_i^{\vartheta_{p_i}(m)}}\equiv \frac{\prod^{m-1}_{i=0}(n+l(m)-i)}{\prod^k_{i=1}p_i^{\vartheta_{p_i}(m)}}~(mod~m).$$

Among the numbers $n,n-1\dots n-m+1$ there are atleast $\lfloor\frac{n}{p^l}\rfloor$ that are diviiable by $p^l$ for every positive integer $l$ and any prime divisor $p$ of $m$.

This is because $$\prod^{m-1}_{i=0}(n-i)=\frac{n!}{(n-m)!}$$ and $$\vartheta_p(a+b)\geq \vartheta_p (a)+\vartheta_p (b).$$

The fraction $\frac{\prod^{m-1}_{i=0}(n-i)}{\prod^k_{i=1}p_i^{\vartheta_{p_i}(m)}}$ can therefore be simplified in such a way that no number of the product $\prod^{m-1}_{i=0}(n-i)$ is divided by $p$ on exponent greater than $\lfloor\log_p m\rfloor$.

In other words the fraction $\frac{\prod^{m-1}_{i=0}(n-i)}{\prod^k_{i=1}p_i^{\vartheta_{p_i}(m)}}$ can be simplified as $\prod^{m-1}_{i=0}\frac{n-i}{\prod^k_{j=1}p_j^{c_j}}$ where for each $j,i$ we have $n-i$ divisable by $p_j^{c_j}$ and $c_j\leq \lfloor\log_{p_j} m\rfloor$ (each factor is an integer).

But then since for every $j$ we have $\prod^k_{j=1}p_j^{c_j}$ divides $\prod^k_{i=1}p_i^{\lfloor\log_{p_i} m\rfloor}$ and since $m\cdot \prod^k_{i=1}p_i^{\lfloor\log_{p_i} m\rfloor}=l(m)$ we have for every $i$ mod $m$,
$$\frac{n+l(m)-i}{\prod^k_{j=1}p_j^{c_j}}=\frac{n-i}{\prod^k_{j=1}p_j^{c_j}}+\frac{l(m)}{\prod^k_{j=1}p_j^{c_j}}=\frac{n-i}{\prod^k_{j=1}p_j^{c_j}}+t\cdot m=\frac{n-i}{\prod^k_{j=1}p_j^{c_j}}$$

This completes the result and hence the length of the minimal period of the sequence, $a_n$ is
$$l(m)=\prod^k_{i=1}p_i^{\lfloor \log_{p_i}m \rfloor+b_i}.$$

\end{proof}

\begin{rem}
If we define $\binom{n}{m}$ also for negative integers $n$, as $\binom{n}{m}=\frac{\prod^{m-1}_{i=0}(n-i)}{m!}$ we can adopt the minimal period length formula for all integers (we can prove in exactly the same way that $\binom{n+l(m)}{m}=\binom{n}{m}$ for every integer $n$).
\end{rem}

The above theorem gives us very easily the following two corollaries:

\begin{cor}
For every positive integer $m=\prod^k_{i=1}p_i^{b_i}$ we have $m^2|l(m)$.
\end{cor}

\begin{cor}
$m$ has only one prime factor ($m=p^b$ where $p$ is prime) if and only if $l(m)=m^2$.
\end{cor}

We donot prove the corollaries here, as it is quite evident that they follow from the previous theorem.

\section{{Main Result}}
Before we prove our main result, we shall state and prove two lemmas.

\begin{lem}
Let $n$ be relatively prime to $m$. Then, $$\binom{n}{m}\equiv \binom{n-1}{m}~(mod~m).$$
\end{lem}

\begin{proof}
Note that if $n$ is relatively prime to $m$ than so is $n-m$. We have $$\binom{n}{m}=\binom{n-1}{m}\cdot\frac{n}{n-m}~(mod~m)$$ which is equivalent to $$(n-m)\cdot \binom{n}{m}=n\cdot \binom{n-1}{m}~(mod~m)$$ which is further equivalent to $$\binom{n}{m}=\binom{n-1}{m}~(mod~m)$$ because $n=n-m~(mod~m)$ and both are relatively prime to $m$.
\end{proof}

\begin{lem}
Let $m$ be even. Then for every integer $k$ we have, $$\binom{m+k}{m}\equiv \binom{l(m)-1-k}{m}~(mod~m).$$
\end{lem}

\begin{proof}
We have $$\binom{l(m)-1-k}{m}=\frac{(l(m)-1-k)(l(m)-1-k-1)\dots(l(m)-1-m-k+1)}{m!}$$ and because there are an even number ($m$) of factors we can multiply each one by $-1$ and still have the same number. So, $$\binom{l(m)-1-k}{m}=\frac{(k+1-l(m))(k+2-l(m))\dots(k+m-l(m))}{m!}$$
which is precisely $\binom{k+m-l(m)}{m}$ and is by the previous theorem equal to $\binom{m+k}{m}~(mod~m)$.
\end{proof}

We state without proof the following famous theorem in Number Theory

\begin{thm}
\textbf{(P.~G.~L.~Dirichlet, 1837)} If $a$ and $b$ are relatively prime positive integers, then the arithmetic progression
$$a,a+b,a+2b,a+3b,\cdots$$
contains infinitely many primes.
\end{thm}

Now, we are in a position to state and proof a result which is stronger than \textbf{Theorem 1.1},

\begin{thm}
A natural number $p>1$ is a prime if and only if $\binom{q}{p}-\lfloor\frac{q}{p}\rfloor$ is divisible by $p$ for every prime $q$, where $\binom{q}{p}$ is the number of different ways in which we can choose $p$ out of $q$ elements and $\lfloor x \rfloor$ is the greatest integer not exceeding the real number $x$.
\end{thm}

\begin{proof}
If $p$ is prime than the result follows from \textbf{Theorem 1.1}. To prove the if only part we just have to construct a counterexample for every composite $p$.

We assume $p$ is an even composite number. We have $\binom{p}{p}=1~(mod~p)$ and by \textbf{Lemma 2.1} we have $\binom{l(p)-1}{p}=\binom{p}{p}=1~(mod~p)$. And by \textbf{Theorem 1.4} we have $\binom{k\cdot l(p)-1}{p}=1$ for every natural $k$.

On the other hand because $p^2|l(p)$ we have $\lfloor\frac{k\cdot l(p)-1}{p}\rfloor=\lfloor\frac{Ap^2-1}{p}\rfloor=-1~(mod~p)$. Therefore $\binom{n}{p}\neq \lfloor\frac{n}{p}\rfloor~(mod~p)$ for every integer $n$ in the sequence $b_k=k\cdot l(p)-1$ and since $1$ and $l(p)$ are relatively prime we have a prime number in this sequence by \textbf{Theorem 2.3}. Therefore there exists a prime number $q$ such that $\binom{q}{p}\neq \lfloor\frac{q}{p}\rfloor~(mod~p)$.

Now we aasume that $p$ is an odd composite number. Now denote $r$ as the smallest prime devisor of $p$ ($r\neq2$). We already know from the proof of \textbf{Theorem 1.1} that $\binom{p+r}{p}\neq \lfloor\frac{p+r}{p}\rfloor$ mod $p$.

Because $r\neq2$ we know that $r+1$ is composite. Note that $p+r+1$ is relatively prime to $p$ because every prime factor that would divide both would also divide $r+1$ and this prime factor would therefore be smaller than $r$ which is a contradiction.

Now by \textbf{Lemma 2.1} we can deduce that $$\binom{p+r+1}{p}\equiv \binom{p+r}{p}~(mod~p)$$ which means $$\binom{p+r+1}{p}\neq \lfloor\frac{p+r+1}{p}\rfloor~(mod~p).$$ Because $p<p+r<p+r+1<2p$ and therefore $\lfloor\frac{p+r}{p}\rfloor=\lfloor\frac{p+r+1}{p}\rfloor$ mod $p$.

Now by \textbf{Theorem 1.4} we know that for every natural $n$ in the sequence $c_k=k\cdot l(p)+(p+r+1)$ we have $\binom{n}{p}\neq \lfloor\frac{n}{p}\rfloor~(mod~p)$.

Since $p$ and $l(p)$ have the same prime factors due to the formula for $l(p)$ so $p+r+1$ which is relatively prime to $p$ is also relatively prime to $l(p)$.

Again by \textbf{Theorem 2.3} we know there exists a prime number in the sequence $c_k$ as defined above. Therefore there exists a prime number $q$ such that $$\binom{q}{p}\neq \lfloor\frac{q}{p}\rfloor~(mod~p).$$

This completes the rest of the proof.
\end{proof}

\section{{Conclusion}}

We see here, that a simple Olympiad problem has lead us to all the above results mentioned in this paper. This is just another testament of the beauty and originality of Olympiad mathematics.

\section{{Acknowledgements}}

The authors would like to thank Prof.~Nayandeep Deka Baruah for his encouragement and for reading through an earlier version of this work. One of us (MPS) would like to thank Prof.~Mangesh B.~Rege for introducing him to the beautiful world of Olympiad mathematics.

\end{document}